\documentclass[twocolumn,noversion,nonote,squeezed,10pt]{cdcarticle}

\usepackage[english]{babel}
\usepackage[T1]{fontenc}
\usepackage{amsthm}
\usepackage{amsmath}
\usepackage{amssymb}
\usepackage{colonequals}
\usepackage[hidelinks]{hyperref}
\usepackage[font=footnotesize,labelfont=bf]{caption}
\usepackage[nameinlink]{cleveref}
\usepackage{cite}
\usepackage{enumitem}
\usepackage{tikz}

\makeatletter
\let\@fnsymbol\@arabic
\makeatother

\Crefname{figure}{Figure}{Figs.}
\crefname{assumption}{Assumption}{Assumptions}
\crefname{problem}{Problem}{Problems}

\newtheorem{assumption}{Assumption}
\newtheorem{problem}{Problem}
\newtheorem{definition}{Definition}
\newtheorem{example}{Example}
\newtheorem{theorem}{Theorem}
\newtheorem{proposition}{Proposition}
\newtheorem{lemma}{Lemma}

\newcommand*{\QEDB}{\hfill\ensuremath{\square}}

\newcommand{\bmat}[1]{\begin{bmatrix}#1\end{bmatrix}}
\newcommand{\pmat}[1]{\begin{pmatrix}#1\end{pmatrix}}
\newcommand{\bmatl}[3][1.2]{\renewcommand{\arraystretch}{#1}\left[\begin{array}{#2}#3\end{array}\right]}
\newcommand{\pmatl}[2]{\renewcommand{\arraystretch}{1.2}\left(\begin{array}{#1}#2\end{array}\right)}

\def\tp{\mathsf{T}}
\def\defeq{\colonequals}

\def\df{\nabla\!f}
\DeclareMathOperator*{\minimize}{minimize}
\def\A{\mathcal{A}}
\def\B{\mathcal{B}}
\def\C{\mathcal{C}}
\def\D{\mathcal{D}}
\def\F{\mathcal{F}}
\def\X{\mathcal{X}}
\def\natural{\mathbb{N}}    
\def\real{\mathbb{R}}       
\def\complex{\mathbb{C}}	
\def\symmetric{\mathbb{S}}  

\title{\LARGE\bf Convex Synthesis of First-Order Methods for Time-Varying\\Smooth Strongly Convex Optimization}

\author{Bryan Van Scoy%
\thanks{\hangindent=5mm%
Department of Electrical and Computer Engineering, Miami University, OH~45056, USA. Email: \texttt{bvanscoy@miamioh.edu}%
}
\and Gianluca Bianchin%
\thanks{\hangindent=5mm%
ICTEAM institute and the Department of Mathematical Engineering (INMA) at the University of Louvain, Belgium. Email: \texttt{gianluca.bianchin@uclouvain.be}\\[2pt]
This material is based upon work supported in part by the National Science Foundation under Award No. 2347121 and in part by the FRFS WEL-T Investigator Programme. Any opinions, findings and conclusions or recommendations expressed in this material are those of the authors and do not necessarily reflect the views of the National Science Foundation.}%
\hspace{-1cm}}

\note{}

\begin{document}

\maketitle

\begin{abstract}
Time-varying optimization is fundamental to decision-making in dynamic environments, where objectives evolve over time due to exogenous signals or data streams.
However, algorithms designed for static problems yield suboptimal decisions in dynamic scenarios, even asymptotically.
In this paper, we develop a robust control synthesis framework to systematically design first-order methods for smooth strongly convex problems that vary in time. Our approach leverages both convex robust control synthesis in the static setting and the internal model principle by directly embedding a model of the underlying variability into the designed algorithm.
\end{abstract}

\section{Introduction}

Optimization algorithms play a central role in modern control, machine learning, and networked systems, underpinning decision-making processes in applications ranging from large-scale data analytics to cyber-physical and energy systems. 
Recent advances have established a systems-and-control perspective for algorithmic optimization analysis and design, whereby iterative optimization methods are modeled as dynamical systems in feedback with nonlinear gradient oracles, enabling their analysis and design via tools from dissipativity theory, passivity, and robust control~\cite{LL-BR-AP:16,BH-LL:17}. 
Building on this viewpoint, convex synthesis approaches~\cite{CS-CE:21} have emerged as a principled framework for the automated design of algorithms, casting algorithm construction as a robust control synthesis problem, enabling the systematic exploration of algorithmic structures beyond classical hand-crafted schemes and leading to provably optimal methods~\cite{CS-CE-TH:23}.

Despite these advances, existing synthesis approaches address \emph{static} optimization problems, while many applications of practical interest involve \emph{time-varying} objectives driven by exogenous signals, evolving environments, or streaming data. 
In such settings, standard first-order methods (e.g., gradient descent) exhibit nonzero steady-state error~\cite{EH:16}, as they lack mechanisms to compensate for persistent temporal variations. 
From a control-theoretic perspective, this limitation stems from the absence of appropriate internal models~\cite{GB-BVS:25-cdc,GB-BVS:26-tac}. 
More critically, accelerated methods designed for time-invariant settings can suffer a noticeable loss of robustness in time-varying regimes, leading to amplified tracking errors and, in some cases, inferior asymptotic performance compared to non-accelerated schemes, as illustrated in \Cref{fig:robustness_comparison}.

\begin{figure}[t]
    \centering\includegraphics[width=\linewidth]{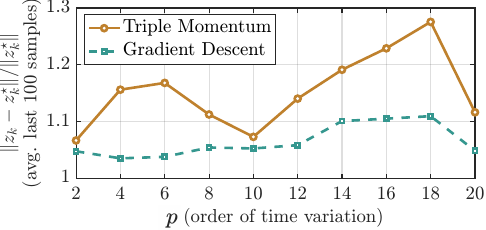}
    \caption{Asymptotic (averaged over the last 100 iterations) relative tracking error versus the order $p$ of the temporal variation, comparing gradient descent and the accelerated triple momentum method~\cite{BVS-RF-KL:17} on the time-varying quadratic problem $f(z,\theta_k)=z^\top Q z + \theta_k^\top z$, with $Q$ random (eigenvalues in $[1,50]$) and $\theta_{k+1}=S\theta_k$, $S \in \real^{p\times p}$. As $p$ increases, the performance of the accelerated method degrades relative to gradient descent, highlighting the need for acceleration schemes that explicitly account for temporal variability.}
    \label{fig:robustness_comparison}
\end{figure}

According to the internal model principle of time-varying optimization~\cite{GB-BVS:25-cdc,GB-BVS:26-tac}, perfect tracking and disturbance rejection require embedding a model of the exogenous signals within the optimization dynamics. Yet, this principle has not been systematically integrated in convex synthesis~\cite{CS-CE:21}.

{\it Contributions:}
Our main contribution is to demonstrate how the convex synthesis procedure from~\cite{CS-CE-TH:2023} can be adapted to incorporate an internal model of the temporal variability of the problem, thereby extending the approach to be applicable to time-varying optimization. Our design is based on the internal model principle of time-varying optimization~\cite{GB-BVS:25-cdc,GB-BVS:26-tac}. The effectiveness of the approach is illustrated through numerical simulations.

{\it Related works:} 
The systems-and-control perspective on optimization interprets first-order methods as dynamical systems in feedback with gradient oracles, thereby enabling analysis through tools from robust control, dissipativity, and integral quadratic constraints (IQCs)~\cite{LL-BR-AP:16,MM-MJ:19,CWS:2022,LL:2022}. 
This viewpoint has also motivated algorithm synthesis methods that go beyond classical hand-crafted designs, including direct synthesis with worst-case convergence guarantees~\cite{LL-PS:20}, IQC-based structure-exploiting design~\cite{SM-CS-CE:21}, linear matrix inequality (LMI)-based synthesis for optimization and saddle-point problems~\cite{DG-CE-CWS:22}, and convex synthesis frameworks for accelerated algorithms and extremum control~\cite{CS-CE:21,TH-CWS:21}. 
In parallel, time-varying optimization has been studied from both online and control-theoretic perspectives. When no explicit model of temporal variability is available, one typically obtains dynamic-regret or input-to-state-stability type guarantees for online gradient, primal-dual, and related running methods~\cite{MZ:03,EH-AA-SK:07,MC-ED-AB:20,AB-ED-AS:19}. 
When the temporal variability can instead be modeled or estimated, prediction-correction and related model-based schemes can achieve substantially improved tracking performance~\cite{AS-AM-AK-GL-AR:16,MF-SP-VP-AR:17,AS-ED-SP-GL-GG:20,NB-RC-SZ:24,GB-JC-JP-ED:21-tcns}. 
More recently, these directions have started to intersect with robust-control tools for time-varying optimization~\cite{AS-ED-SP-GL-GG:20,GB-BVS:26-tac}. 
At a more fundamental level, the internal-model viewpoint connects exact asymptotic tracking in time-varying optimization to output regulation theory, showing that exact tracking requires embedding a suitable model of the temporal variability within the algorithm itself~\cite{BF-WW:76,JH-WW:84,NB-RC-SZ:24,GB-BVS:26-tac}. 
Particularly relevant to the design methodology pursued in this work are~\cite{DA-LP-LM:2022,FDP:1993}, where it is observed that output regulation of nonlinear models may be achieved by linear controllers, provided that they incorporate more dynamics than those generated by the exosystem. These developments motivate the present work: rather than treating convex synthesis and internal-model design as separate threads, we integrate them into a unified synthesis framework to design optimization algorithms tailored to structured time-varying problems.

{\it Notation:} 
We denote the natural numbers including zero by $\natural$, and the reals by $\real$. The set of $n\times n$ symmetric matrices is $\symmetric^n$. We denote the state-space realization $(A,B,C,D)$ of a discrete-time LTI system $G$ as
\[
    G = \bmatl[1]{c|c}{A & B \\ \hline C & D}.
\]

\section{Problem Setup}

Consider an unconstrained optimization problem in which the objective function depends on a time-varying parameter:
\begin{equation}\label{eq:tvopt}
    \minimize \ f(z,\theta_k),
\end{equation}
where $f : \real^d\times\real^p\to\real$ is the objective function, ${z\in\real^d}$ the decision variable, and $\theta : \natural\to\real^p$ a time-varying parameter. We make the following assumption on the objective function.

\begin{assumption}[Objective]\label{assumption:objective}
    For any fixed parameter $\theta_\circ\in\real^p$, the function $z\mapsto f(z,\theta_\circ)$ is $L$-smooth and $\mu$-strongly convex with known parameters $0<\mu<L$.
\QEDB\end{assumption}

The objective being $L$-smooth means that the gradient is Lipschitz continuous with parameter $L$, and the objective being $\mu$-strongly convex means that $z\mapsto f(z,\theta_\circ)-\tfrac{\mu}{2} \|z\|^2$ is convex. With this assumption, the optimization problem has a unique optimal trajectory $z^* : \natural\to\real^d$. We use $\F_{\mu,L}$ to denote the set of all functions that satisfy \Cref{assumption:objective}.

The optimization problem varies in time due to the parameter $\theta_k$, which we assume is generated by a (known) LTI system.

\begin{assumption}[Exosystem]\label{assumption:exosystem}
    The parameter signal $\theta : \natural\to\real^p$ is generated by an autonomous LTI system satisfying
    \begin{equation}\label{eq:exosystem}
        \theta_{k+1} = S \theta_k \quad\text{for all }k\in\natural
    \end{equation}
    and some initial condition $\theta_0\in\real^p$. Moreover, all the eigenvalues of $S$ have unit modulus, all exosystem trajectories are bounded, and the model $S$ is known.~\QEDB
\end{assumption}

The requirement that~\eqref{eq:exosystem} is known is motivated by~\cite{GB-BVS:25-cdc,GB-BVS:26-tac}, which show that knowledge of the exosystem is necessary to compute an exact optimizer of~\eqref{eq:tvopt}. In its absence, one must incur nonzero regret~\cite{AS-ED-SP-GL-GG:20}; here, we focus on algorithms that achieve zero regret, and hence impose this requirement.

\begin{figure}[t]
    \centering
    \scalebox{0.8}{\includegraphics{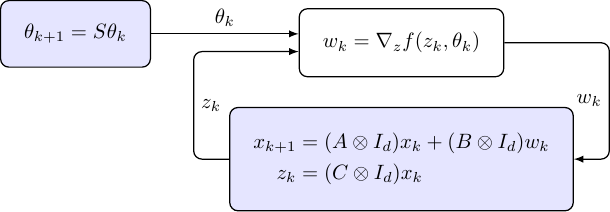}}
    \caption{Block diagram of a first-order method for time-varying optimization. The exosystem generates the time-varying parameter $\theta_k$, which affects the gradient of the objective. The method is a dynamical system with state $x_k$ that generates a sequence of points $z_k$ at which to evaluate the gradient based on measurements $w_k$ of the time-varying gradient.}
    \label{fig:alg}
\end{figure}

In this work, we are interested in designing first-order optimization methods that can query the gradient of the objective at each iteration. Specifically, the method selects a point $z_k$ at which to query the gradient, and then receives the corresponding value of the gradient
\begin{equation}
    w_k = \nabla_z f(z_k,\theta_k),
\end{equation}
with $\nabla_z f$ denoting the gradient of $f$ with respect to its first 
argument. To exclude exosystem modes that do not influence the gradient (and are therefore irrelevant to the optimization), we assume the system contains no nontrivial zero dynamics.

\begin{assumption}[Absence of zero dynamics]\label{assumption:zero-dynamics}
    The objective $f$ and exosystem trajectory $\theta$ are such that $\nabla_z f(z,\theta_k)=0$ for all $k\in\natural$ and some $z\in\real^d$ implies $\theta_k = 0$ for all $k\in\natural$.
\QEDB\end{assumption}

An optimization algorithm then consists of a mapping from $w$ to $z$. For implementability, we require this mapping to be causal, linear, time-invariant, strictly proper, and to act homogeneously across each dimension of $\mathbb{R}^d$. Formally, we consider methods of the form
\begin{subequations}\label{eq:alg}
\begin{align}
    x_{k+1} &= (A\otimes I_d) x_k + (B\otimes I_d) w_k, \\
    z_k &= (C\otimes I_d) x_k,
\end{align}
\end{subequations}
for some state-space matrices $A\in\real^{n\times n}$, $B\in\real^{n\times 1}$, and $C\in\real^{1\times n}$, where $x_k\in\real^{nd}$ is the state. In line with~\cite[\S 2.2]{CS-CE:2021}, we assume that the pair $\left(A, C\right)$ is detectable without loss of generality. The algorithm is illustrated in \Cref{fig:alg}.

For the algorithm~\eqref{eq:alg} to solve time-varying optimization problems, it must in particular be capable of solving time-invariant problems (e.g., when the objective does not depend on $\theta_k$). A necessary condition for this is that the pair $(A,C)$ has an observable eigenvalue at one~\cite[\S~II.C]{CS-CE-TH:2023}. This, in turn, implies the existence of a fixed point.

\begin{assumption}[Fixed point]\label{assumption:equilibirum-point}
    There exist $x_* \in \real^{nd}$ such that 
    \[
        x_* = (A\otimes I_d) x_* \qquad\text{and}\qquad (C\otimes I_d) x_*\neq 0. \tag*{\QEDB}
    \]
\end{assumption}

\begin{definition}
    We say that algorithm~\eqref{eq:alg} \textit{(locally) asymptotically tracks the optimizer} of~\eqref{eq:tvopt} with rate $\rho\in(0,1)$ if there exist constants $c>0$ and $\delta>0$ such that, for any initial condition $x_0$ with $\|x_0 - x_*\| < \delta$,
    \begin{equation}\label{eq:exp_rate}
        \|z_k - z_k^\star\| \leq c\,\rho^k\,\|x_0-x_0^\star\| \quad\text{for all }k\in\natural,
    \end{equation}
    where $z_k^\star\in\real^d$ is the unique optimizer at each $k$ and $x_k^\star$ is the corresponding fixed point of~\eqref{eq:alg}. Moreover, we say that the tracking is \textit{robust with respect to a set of objective functions and exosystem trajectories} if the bound holds for all such instances.
\end{definition}

With this setup, we can now state our main goal.

\begin{problem}\label{prob:main}
    Design an efficient procedure that, given the objective function parameters $\mu$ and $L$ and the temporal variability model $S$, automatically selects the state-space matrices $(A, B, C)$ in~\eqref{eq:alg} such that the resulting method robustly asymptotically tracks the optimizer of~\eqref{eq:tvopt} with optimal (i.e., minimal) rate $\rho$. Optimality is with respect to all admissible choices of state-space matrices $(A,B,C)$, and robustness is with respect to all objective functions $f$ and exosystem trajectories $\theta$ that satisfy \Cref{assumption:objective,assumption:exosystem,assumption:zero-dynamics}.~\QEDB
\end{problem}

\section{Internal Model Embedding}

In this section, we derive a structure for~\eqref{eq:alg} to ensure that the algorithm achieves asymptotic tracking. This will then enable us to reformulate~\Cref{prob:main} as a robust control problem.

Achieving asymptotic tracking of the optimizer imposes specific structural requirements on the method. In the time-invariant case (in which the model of the time variation is $S = I$, so that $\theta_k$ is constant), it is well-known that the method~\eqref{eq:alg} must have the form of an LTI system in series with an integrator~\cite[\S 2.2]{CS-CE:2021}. When the objective varies in time, by the \textit{internal model principle of time-varying optimization}~\cite{GB-BVS:25-cdc,GB-BVS:26-tac}, this concept generalizes to the requirement that the method must embed knowledge of the exosystem in its dynamics. 

\begin{figure}[t]
\centering
\includegraphics{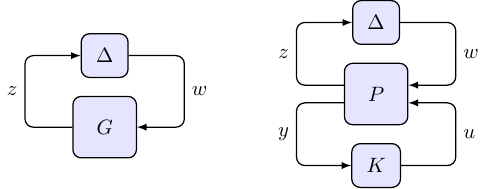}
\caption{Block diagrams of system interconnections. (Left) The algorithm transfer function $G(z)$ in feedback with an uncertainty $\Delta$ as in Thm.~\ref{thm:main}. (Right) Plant $P$ in feedback with controller $K$ and uncertainty $\Delta$ as in~\eqref{eq:plant}.}
\label{fig:robust-control}
\end{figure}

Building on this finding, we will now present the main result of this work. To state our result, we let $\mathbf{\Delta}$ denote the set of static maps $\Delta$ such that $w=\Delta(z)$ if and only if $w_k = \df(z_k)$ for some function $f\in\F_{\mu,L}$ with $\df(0) = 0$. Moreover, given a multi-index $\alpha \in \mathbb{N}^p$ and eigenvalues $\lambda = (\lambda_1,\ldots,\lambda_p)$ of $S$, we define $\lambda^\alpha \defeq \lambda_1^{\alpha_1} \lambda_2^{\alpha_2} \cdots \lambda_p^{\alpha_p}.$ We then define the set of \textit{harmonics} associated with $\lambda$ as
\begin{equation}
    \Omega_\lambda \defeq \{ \lambda^\alpha \mid \alpha \in \mathbb{N}^p \}.
\end{equation}
That is, $\Omega_\lambda$ contains all possible products of the eigenvalues of $S$, including repeated products, and characterizes all modes that can arise from their combinations.

\begin{example}\label{example:harmonics}
    If the exosystem contains a single frequency $\theta = \pi/3$, then $S$ has eigenvalues $\lambda = (e^{j\pi/3},e^{-j\pi/3})$, in which case the set of harmonics is $\Omega_\lambda = \{0,\tfrac{\pi}{3},\tfrac{2\pi}{3},\pi\}$.
\end{example}

\begin{theorem}[Internal model-based optimization design]\label{thm:main}
    Let $0<\mu<L$ and $\rho\in(0,1)$. Let $(A,B,C)$ be the state-space matrices of an algorithm~\eqref{eq:alg}. Then, the following statements are equivalent:
    \begin{itemize}
        \item[(i)] Algorithm~\eqref{eq:alg} locally asymptotically tracks the optimizer of~\eqref{eq:tvopt} with rate $\rho$ for all functions $f$ and exosystem trajectories $\theta$ that satisfy \Cref{assumption:objective,assumption:exosystem,assumption:zero-dynamics}.
        \item[(ii)] There exists a proper rational transfer function $K(z)$ such that $G(z) = C (zI - A)^{-1} B$ factors as
        \begin{equation}\label{eq:structure}
            G(z) = K(z)\,H(z) \text{ where } H(z)\defeq\frac{z^{|\Omega_\lambda|-1}}{\prod_{\omega\in\Omega_\lambda} (z-\omega)},
        \end{equation}
        and the feedback interconnection of algorithm~\eqref{eq:alg} with uncertainty $\Delta$ (see \Cref{fig:robust-control}) is robustly exponentially stable\footnote{Robust exponential stability with rate $\rho\in(0,1)$ amounts to ensuring that there exists $c>0$ such that, for all $\Delta\in\mathbf{\Delta}$ and all $x_0,$ the closed-loop state $x_k$ satisfies $\|x_k\|\leq c\,\rho^k\|x_0\|$ for all $k\geq 0$.} with rate $\rho$ for all~${\Delta\in\mathbf{\Delta}}$. \QEDB
    \end{itemize}
\end{theorem}

The proof of this claim is postponed to the appendix. 

It follows from \Cref{thm:main} that designing an algorithm~\eqref{eq:alg} to track the optimizer of the time-varying optimization problem~\eqref{eq:tvopt} is equivalent to solving a robust control problem with plant $P$ and uncertainty set $\mathbf{\Delta}$. From a control-theoretic viewpoint, the structural condition~\eqref{eq:structure} admits the following interpretation: the transfer function $H(z)$ captures the required internal model through its pole structure and is thus regarded as the \textit{plant}. In turn, $K(z)$ shapes the behavior of $H(z)$ and is therefore interpreted as the \textit{controller}.

Let a state-space realization of the plant be
\begin{equation}\label{eq:model}
    H = \bmatl{c|c}{A_p & B_p \\ \hline C_p & 0}.
\end{equation}
By \Cref{thm:main}, the algorithm must then have the form shown in \Cref{fig:robust-control} (right) with plant
\begin{equation}\label{eq:plant}
    \bmat{z \\ y} = P \bmat{w \\ u} \quad\text{where}\quad
    P = \bmatl{c|cc}{A_p & B_p & 0 \\ \hline 0 & 0 & 1 \\ C_p & 0 & 0}.
\end{equation}
Let $n_p \defeq |\Omega_\lambda|$ denote the dimension of the plant, which is the number of exosystem harmonics.

\section{Robust Control-Based Design}

With \Cref{prob:main} reformulated as a robust control problem, we now propose the use of the convex synthesis methodology from \cite{CS-CE-TH:2023,CS-CE:2025} to design the controller $K(z)$. Once the controller is obtained, the algorithm transfer function is given in~\eqref{eq:structure}, the state-space matrices $(A,B,C)$ are any realization of this transfer function, and the algorithm is~\eqref{eq:alg}.

\subsection{Transformations}

Following this strategy, we first apply a chain of transformations to put the system into standard form~\cite{CS-CE-TH:23}. As a first transformation, we exponentially weight the signals so that stability of the transformed system is equivalent to exponential stability of the original system. To that end, define the exponential signal weighting map
\[
    T_\rho(z_0,z_1,z_2,\ldots) = (z_0,\rho z_1, \rho^2 z_2,\ldots)
\]
with parameter $\rho>0$. The map $T_\rho$ is linear and invertible with inverse $T_{\rho^{-1}}$. The transformed system is shown in \Cref{fig:weighting}, where the transformed plant is
\begin{equation}
    \bmat{\bar z \\ \bar y}
    = \bmatl[1]{c|cc}{\rho^{-1} A_p & \rho^{-1} B_p & 0 \\ \hline 0 & 0 & 1 \\ C_p & 0 & 0} \bmat{\bar w \\ \bar u}.
\end{equation}
The transformed system is Lyapunov stable\footnote{Exponential stability means there exists $c>0$ such that, for every $\bar x_0$, $\|\bar x_k\|\leq c\,\rho^k\|\bar x_0\|$ for all $k$. Lyapunov stability is the same with $\rho=1$.} if and only if the original system is exponentially stable with rate $\rho$~\cite{CS-CE-TH:23}.

We next filter the input and output of the (transformed) uncertainty through a static filter to produce the signals $\bar p$ and~$\bar q$. The system from $(\bar q,\bar u)$ to $(\bar p,\bar y)$ is
\[
    \bmat{\bar p \\ \bar y}
    = \bmatl[1]{c|cc}{ \rho^{-1} A_p & \rho^{-1} B_p & \rho^{-1} \mu B_p \\ \hline 0 & -1 & L-\mu \\ C_p & 0 & 0 } \bmat{\bar q \\ \bar u}.
\]
All admissible trajectories of the transformed system satisfy the passivity property $0 \leq \sum_{k=0}^T \bar q_k^\tp \bar p_k$ for all $T\in\natural$~\cite{CS-CE-TH:23}. In fact, we can generate a parameterized set of constraints by filtering~$\bar p$ through the following system, known as a Zames--Falb multiplier~\cite{GZ-PF:1968},
\[
    \setlength{\arraycolsep}{3pt}
    \psi = \bmatl{c|c}{A_f & B_f \\ \hline C_f & D_f}
    = \bmatl[1]{cccc|c}{0 & 1 & \ldots & 0 & 0 \\[-4pt] 0 & 0 & \ddots & \vdots & \vdots \\[-4pt] \vdots & \vdots & \ddots & 1 & 0 \\ 0 & 0 & \ldots & 0 & 1 \\\hline \lambda_\ell & \lambda_{\ell-1} & \ldots & \lambda_1 & \lambda_0}.
\]
The output $\bar r$ of this filter also satisfies the passivity property $0\leq\sum_{k=0}^T \bar q_k^\tp \bar r_k$ for all $T\in\natural$ if the filter parameters $\lambda_0,\ldots,\lambda_\ell\in\real$ belong to the set
\begin{equation*}
    \Lambda_\ell^\rho = \biggl\{(\lambda_0,\ldots,\lambda_\ell) \mid \lambda_1,\ldots,\lambda_\ell \leq 0, \sum_{j=0}^\ell \rho^{-j} \lambda_j \geq 0 \biggr\}.
\end{equation*}
The transformed plant combined with the filter yield
\[
    \bmat{\bar r \\ \bar y} = \bmatl{cc|cc}{A_f & 0 & -B_f & (L-\mu) B_f \\ 0 & \rho^{-1} A_p & \rho^{-1} B_p & \rho^{-1} \mu B_p \\ \hline C_f & 0 & -D_f & (L-\mu) D_f \\ 0 & C_p & 0 & 0} \bmat{\bar q \\ \bar u}.
\]
For notational convenience, we denote this system by
\[
    \hat P = \bmatl{c|cc}{\hat A & \hat B_w & \hat B \\ \hline \hat C_z & \hat D_{zw} & \hat D_z \\ \hat C & 0 & 0}.
\]
To summarize, for any transformed uncertainty $\bar\Delta$ with $\Delta\in\mathbf{\Delta}$, all trajectories of the transformed system in \Cref{fig:weighting} satisfy the passivity property $0\leq\sum_{k=0}^T \bar q_k^\tp \bar r_k$ for all $T\in\natural$ whenever the filter parameters satisfy $\lambda\in\Lambda_\ell^\rho$.

\begin{figure}
\centering
\includegraphics{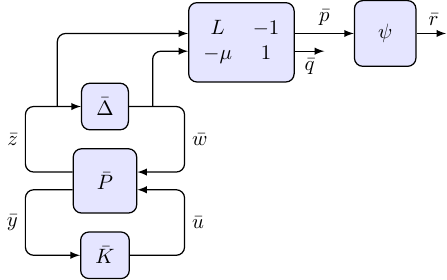}
\caption{Equivalent robust control setup after exponential scaling and filtering. Signals are scaled by $T_{\rho^{-1}}$, e.g., $\bar u = T_{\rho^{-1}} u$, and the plant is transformed as $\bar P = T_{\rho^{-1}} \circ P \circ T_\rho$, and similarly for the uncertainty and controller. The input $\bar z$ and output $\bar w$ of the transformed uncertainty $\bar\Delta$ are then passed through a static map, with the first output $\bar p$ then passed through a dynamic filter $\psi$ resulting in $\bar r$.}
\label{fig:weighting}
\end{figure}

\subsection{Analysis}

Before designing controllers, we begin with the simpler problem of analyzing the convergence rate of a fixed controller. Let $K$ be a controller of dimension $n_c$, and let $(\A,\B,\C,\D)$ denote a state-space realization of the corresponding system from $\bar q$ to $\bar r$ in \Cref{fig:weighting}. We then have the following result.

\begin{proposition}[Analysis]
If there exists a matrix $\X\in\symmetric^{n_p+\ell+n_c}$ and filter parameters $\lambda\in\Lambda_\ell^\rho$ such that
\begin{equation}\label{eq:analysis}
    \X\succ 0, \quad \pmat{\A & \B \\ I & 0 \\ \C & \D \\ 0 & 1}^\tp \pmat{\X & 0 & 0 & 0 \\ 0 & -\X & 0 & 0 \\ 0 & 0 & 0 & 1 \\ 0 & 0 & 1 & 0} \pmat{\A & \B \\ I & 0 \\ \C & \D \\ 0 & 1} \prec 0,
\end{equation}
then the interconnection in \Cref{fig:robust-control} (right) is robustly exponentially stable with rate $\rho$ for all uncertainties~${\Delta\in\mathbf{\Delta}}$.~\QEDB
\end{proposition}

\subsection{Convex synthesis with fixed multiplier}

While the analysis condition enables us to evaluate the rate of a given controller, our goal is to design controllers. Simply searching over controller matrices is not tractable, as the analysis conditions are not simultaneously convex in $\X$, $\Lambda_\ell^\rho$, and the state-space matrices of $K$. When the multiplier $\psi$ is fixed, however, the controller synthesis follows from seminal $\mathcal{H}_\infty$ synthesis results~\cite{PG-PA:1994}. To that end, let
\[
    \hat U = \text{diag}(I_\ell,{C_p}_\perp,1), \quad
    \hat V^\tp = \bmat{B_f^\tp & \tfrac{\mu}{\rho(L-\mu)} B_p^\tp & D_f^\tp}_\perp
\]
where $M_\perp$ denotes a matrix whose columns are a basis for the kernel of $M$. Then, there exists a controller $K$ and matrix $\X\succ 0$ such that the analysis LMI~\eqref{eq:analysis} is feasible if and only if there exist symmetric matrices $\hat X,\hat Y\in\symmetric^{n_p+\ell}$ such that
\begin{subequations}
\bgroup\setlength{\arraycolsep}{2pt}
\begin{align*}\label{eq:synthesis-fixed-multiplier}
    \bigl[\cdot\bigr]^\tp \pmatl{cc|cc}{\hat X & 0 & 0 & 0 \\ 0 & -\hat X & 0 & 0 \\ \hline 0 & 0 & 0 & 1 \\ 0 & 0 & 1 & 0} \left[\pmat{\hat A & \hat B_w \\ I & 0 \\ \hline \hat C_z & \hat D_{zw} \\ 0 & 1}\hat U\right] &\prec 0, \\
    \left[\hat V \pmatl{cc|cc}{-I & \hat A & 0 & \hat B_w \\ 0 & \hat C_z & -1 & \hat D_{zw}}\right] \pmatl{cc|cc}{\hat Y & 0 & 0 & 0 \\ 0 & -\hat Y & 0 & 0 \\ \hline 0 & 0 & 0 & 1 \\ 0 & 0 & 1 & 0} \bigl[\cdot\bigr]^\tp &\succ 0, \\
    \pmat{\hat Y & I \\ I & \hat X} &\succ 0.
\end{align*}\egroup
\end{subequations}
The LMIs have the following interpretation with respect to the transformed plant $\hat P$: the first LMI is the primal dissipativity constraint, the second is the dual dissipativity constraint, and the third is the coupling condition~\cite{PG-PA:1994,JCW:1972}.

\subsection{Convex synthesis over fixed multipliers}

The matrix inequalities in the previous section are linear in the unknowns when the multiplier $\psi$ is fixed, but become nonlinear when the multiplier coefficients are included as variables. While there is no known method to efficiently search over general multipliers, a particularly clever approach was discovered in~\cite{CS-CE-TH:2023} for the case of causal Zames--Falb multipliers. The convexification procedure hinges on the fact that SISO LTI systems commute, with their state-space realizations related through a particular state transformation. The convex algorithm synthesis based on this insight adapted to our particular setting as follows. 

Given filter parameters $\lambda\in\Lambda_\ell^\rho$, the state transformation relating the commuting LTI system is
\[
    T = \bmat{(L-\mu)^{-1} I_\ell & 0 \\ (L-\mu)^{-1} N & \sum_{j=0}^\ell \lambda_j \rho^j A_p^{-j}} = \bmat{T_f \\ \tilde T},
\]
where the matrix $N\in\real^{n_p\times\ell}$ is the solution to the Sylvester equation
\begin{equation}\label{eq:N}
    A_p N - \rho N A_f + \mu B_p C_f = 0.
\end{equation}
With the annihilator matrix $V^\tp = \bmat{\rho^{-1} \mu B_p^\tp & L-\mu}_\perp$, consider the following matrix inequalities:
\begin{subequations}\label{eq:synthesis}
\bgroup\setlength{\arraycolsep}{1.5pt}
\begin{align}
    \bigl[\cdot\bigr]^\tp \pmatl{cc|cc}{\hat X & 0 & 0 & 0 \\ 0 & -\hat X & 0 & 0 \\ \hline 0 & 0 & 0 & 1 \\ 0 & 0 & 1 & 0} \left[\pmat{\hat A & \hat B_w \\ I & 0 \\ \hline \hat C_z & \hat D_{zw} \\ 0 & 1}\hat U\right] &\prec 0, \\
    \left[V \pmatl{cc|cc}{-I & A_p & 0 & \tilde T \hat B_w \\ 0 & 0 & -1 & \hat D_{zw}}\right] \pmatl{cc|cc}{\tilde Y & 0 & 0 & 0 \\ 0 & -\rho^{-2}\tilde Y & 0 & 0 \\ \hline 0 & 0 & 0 & 1 \\ 0 & 0 & 1 & 0} \bigl[\cdot\bigr]^\tp &\succ 0, \\
    \pmat{\tilde Y & \tilde T \\ \tilde T^\tp & \hat X} &\succ 0.
\end{align}\egroup
\end{subequations}
We can now state the main convex synthesis result from~\cite{CS-CE-TH:2023}.

\begin{theorem}\label{thm:synthesis}
    Fix the rate $\rho\in(0,1)$, number of multipliers $\ell\in\natural$, and model $S\in\real^{p\times p}$. The following are equivalent.
    \begin{enumerate}
        \item There exists an algorithm~\eqref{eq:alg} such that the analysis LMI~\eqref{eq:analysis} is feasible for some $\X$ and $\lambda\in\Lambda_\ell^\rho$.
        \item There exist matrices $N\in\real^{n_p\times\ell}$, $\hat X\in\symmetric^{n_p+\ell}$, and $\tilde Y\in\symmetric^{n_p}$ as well as parameters $\lambda\in\Lambda_\ell^\rho$ that satisfy the Sylvester equation~\eqref{eq:N} and LMIs~\eqref{eq:synthesis}.
  \QEDB  \end{enumerate}
\end{theorem}

The (equivalent) statements in \Cref{thm:synthesis} imply that there exists a controller $K$ for which the corresponding algorithm~\eqref{eq:structure} exponentially tracks the optimizer of~\eqref{eq:tvopt} for all functions $f$ and trajectories $\theta$ that satisfy \Cref{assumption:objective,assumption:exosystem,assumption:zero-dynamics}. Moreover, \eqref{eq:N} and \eqref{eq:synthesis} are \textit{convex} constraints in all of the unknowns, and hence provide an efficient procedure to verify the existence of an algorithm that exponentially tracks the optimizer with a given rate. The minimal rate can then be found by bisection over $\rho$. The corresponding controller can be reconstructed from the multiplier by finding a solution to the analysis condition~\eqref{eq:analysis}, thereby solving \Cref{prob:main}.


\section{Numerical Investigation}

We now use the convex synthesis procedure to simulate the synthesized algorithm on a particular time-varying optimization problem and study how the rate depends on parameters.

\subsection{Simulation}\label{sec:simulation}

We first illustrate our results by applying our designed algorithm to a particular time-varying optimization problem. We consider the instance of~\eqref{eq:tvopt} with objective
\begin{align}\label{eq:logistic}
    f(z,\theta_k) = \tfrac{1}{2} \bigl(z - e_1^\tp\theta_k\bigr)^2 + a \log\bigl(1+e^{b z}\bigr),
\end{align}
where $f: \real \times \real^2 \to \real$ and $e_1 = (1,0,\ldots,0)\in\real^p$. This models a logistic regression problem with a time-varying regularization term. Intuitively, an optimizer is a point that tracks the time-varying signal $e_1^\tp\theta_k$ while avoiding large values, which are penalized by the logistic term. This objective satisfies \Cref{assumption:objective} with parameters $\mu=1$ and $L=1+\tfrac{ab^2}{4}$. For our experiments, we used $a = 1$ and $b = 6$.
We generated the parameter vector $\theta_k$ with frequencies $\theta=0$ and $\theta=\pi/3$. The harmonics are those in \Cref{example:harmonics}, and the transfer function in~\eqref{eq:structure} is $H(z) = z^5/(z^6 - 1)$. The transfer function of the synthesized controller (after reduction) is $K(z) = -0.18/z^5$. Interestingly, the controller canceled all five zeros at $z=0$ of $H(z)$, resulting in the closed-loop transfer function $G(z) = -0.18 / (z^6-1)$, which exponentially tracks the time-varying optimizer of~\eqref{eq:logistic} with rate $\rho = 0.9672$. In \Cref{fig:simulation}, we plot the gradient norm $\|w_k\|$ of the algorithm trajectory along with the theoretical rate $\rho^k$, illustrating that the synthesized algorithm attains the expected rate. 

\begin{figure}[t]
\centering\includegraphics[width=0.9\columnwidth]{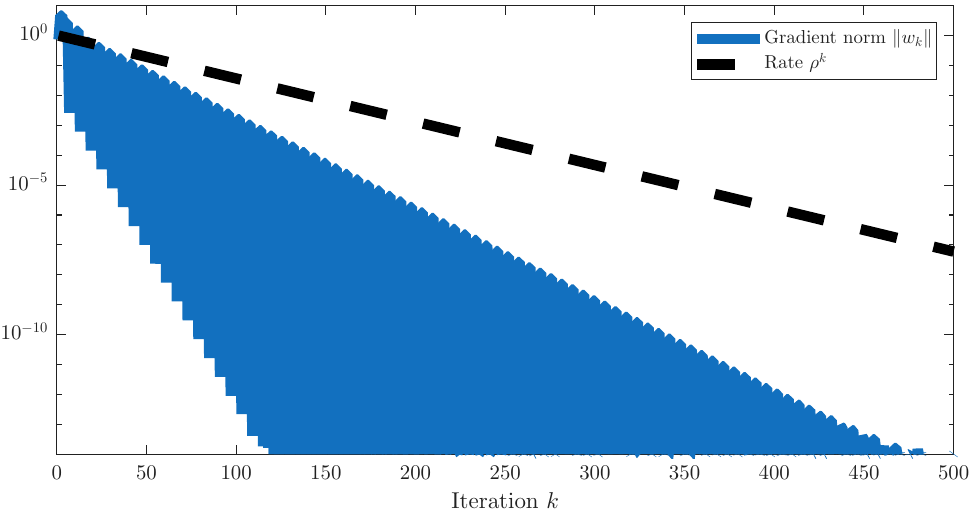}
\caption{Simulation of the synthesized controller on a time-varying regularized logistic regression problem; see \Cref{sec:simulation}.}
\label{fig:simulation}
\end{figure}

\subsection{Convergence rates}\label{sec:rate}

The convergence rate of the algorithm depends on the objective parameters $\mu$ and $L$, the exosystem $S$, and the length $\ell$ of the multiplier used to characterize the gradient. To better understand how the rate varies with these parameters, we plot~$\rho$ as a function of the exosystem frequency $\theta$ in \Cref{fig:rate} with $\mu=1$, $L=10$, and $\ell=1$. The controller is then synthesized using the harmonics $\Omega_\lambda$, where $\lambda = (e^{j\theta},e^{-j\theta})$ if $0<\theta<\pi$, $\lambda = 1$ if $\theta=0$, and $\lambda = -1$ if $\theta = \pi$. Also shown is the rate of the triple momentum method~\cite{BVS-RF-KL:17}, which achieves the optimal rate $\rho_\text{TM} = 1-\sqrt{\mu/L}$ in the time-invariant case when $\theta=0$.

\begin{figure}[ht]
\centering\includegraphics[width=0.9\linewidth]{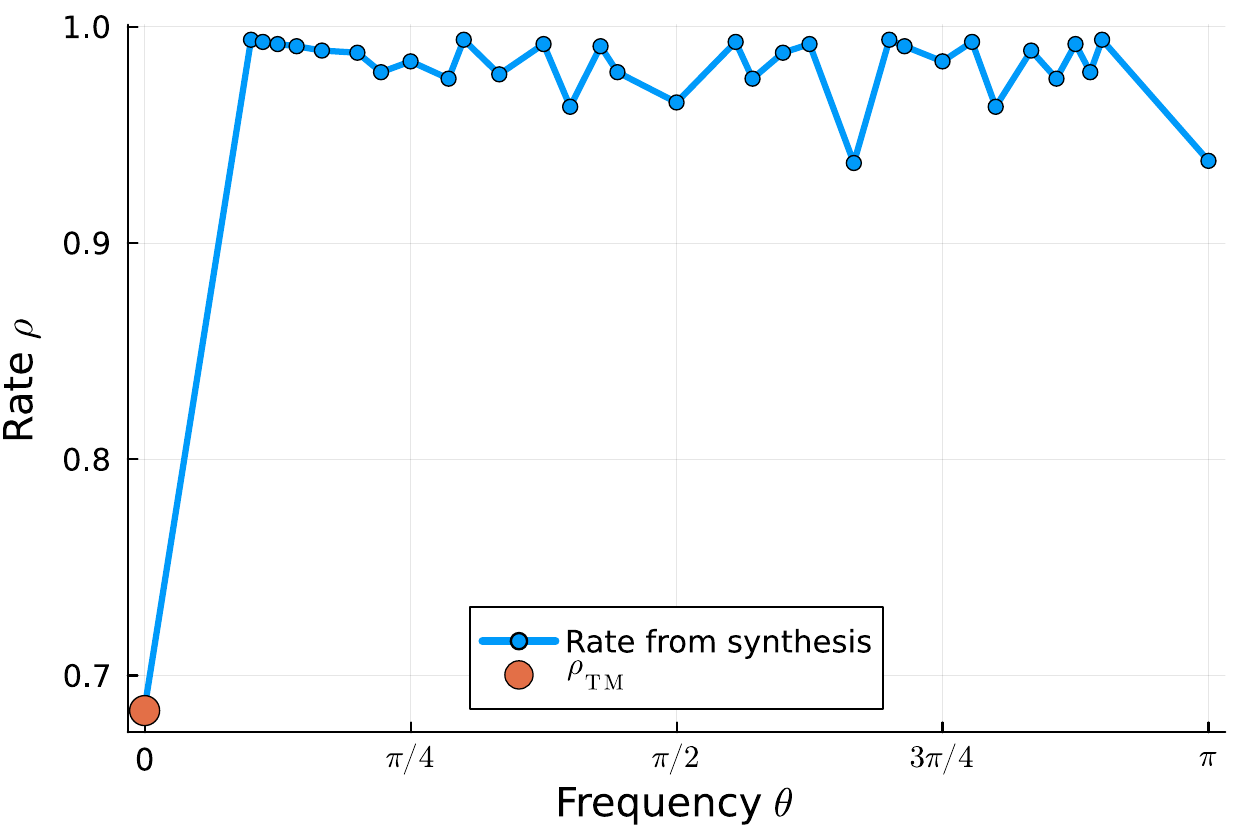}
\caption{Rate $\rho$ of the synthesized algorithm as a function of exosystem frequency $\theta$ with $\mu=1$, $L=10$, and $\ell=1$.}
\label{fig:rate}
\end{figure}

We now draw several observations from \Cref{fig:rate}:
\begin{itemize}
    
    \item The attainable rate varies with the frequency $\theta$ of the time variation. More precisely, the rate is not monotonic in $\theta$, meaning that in some cases there exists an algorithm that converges faster for problems with higher frequencies than for lower frequencies.
    
    \item Developing further on the point above, the fact that the attainable rate is not constant in $\theta$ reveals a key difference from~\cite{BV-GB:2026}, where, for quadratic objectives, lower bounds on the convergence rate do \textit{not} depend on the frequency of the temporal variation, but only on the order of the exosystem model.
    
    \item The rate is significantly slower than the time-invariant counterpart (where 
    $\rho = \rho_\text{TM} \approx 0.6838$), even for slow time variations (small $\theta>0$). 
\end{itemize}

\section{Conclusions}

We showed how internal models can be systematically embedded into the convex synthesis framework from robust control to design first-order optimization methods for smooth, strongly convex objectives that vary in time. The proposed methodology is constructive and directly implementable whenever a model of the temporal variability is available, thereby enabling principled algorithm design with guaranteed tracking and robustness properties. Extensions include nonlinear exosystem models and adaptive algorithms that estimate the model online.

\appendix

\section{Proof of \Cref{thm:main}}
\label{sec:comlemetary_proofs}

The proof of \Cref{thm:main} relies on two lemmas.

\begin{lemma}[Internal model principle]\label{lemma:IMP}
    Fix the objective parameters $\mu$ and $L$, exosystem model $S$, algorithm matrices $(A,B,C)$, and rate $\rho\in(0,1)$. Suppose the feedback interconnection of algorithm~\eqref{eq:alg} with uncertainty $\Delta$ (see \Cref{fig:robust-control} (left)) is robustly locally exponentially stable with rate $\rho$ for all~${\Delta\in\mathbf{\Delta}}$. Then, the following statements are equivalent:
    \begin{itemize}
        \item[(i)] Algorithm~\eqref{eq:alg} locally asymptotically tracks the optimizer of~\eqref{eq:tvopt} with rate $\rho$ for all functions $f$ and exosystem trajectories $\theta$ that satisfy \Cref{assumption:objective,assumption:exosystem,assumption:zero-dynamics}.
        \item[(ii)] There exists an analytic mapping $\sigma : \real^p\to\real^n$ satisfying the internal model kinematic and output equations:
        \begin{equation*}
        \sigma(S\theta) = A \sigma(\theta) \quad\text{and}\quad
            \nabla_z f(C \sigma(\theta), \theta) = 0.
        \end{equation*}
    \end{itemize}
\end{lemma}

\begin{proof}
    The result follows from \cite[Theorem 2]{GB-BVS:25-cdc}.\hfill
\end{proof}

We now explore some consequences of condition (ii). Since $\sigma$ is real analytic, it admits a (locally convergent) multivariate power series expansion of the form
\begin{equation}\label{eq:series}
    \sigma(\theta) = \sum_{\alpha\in\natural^p} c_\alpha \theta^\alpha,
    \qquad
    \theta^\alpha \defeq \theta_1^{\alpha_1}\cdots \theta_p^{\alpha_p},
\end{equation}
where the coefficients $c_\alpha \in \real^n$ are uniquely determined. The monomials $\{\theta^\alpha\}_{\alpha\in\natural^p}$ form a linearly independent basis for analytic functions, so this representation is unique.

If $S$ is diagonalizable with eigenvalues $\lambda=(\lambda_1,\ldots,\lambda_p)$, then in diagonalized coordinates one has $(S\theta)^\alpha = \lambda^\alpha \theta^\alpha$, where $\lambda^\alpha \defeq \lambda_1^{\alpha_1}\cdots\lambda_p^{\alpha_p}$. Substituting the expansion into the kinematic relation $\sigma(S\theta)=A\sigma(\theta)$ yields
\[
    \sum_{\alpha\in\natural^p} \lambda^\alpha c_\alpha \theta^\alpha
    =
    \sum_{\alpha\in\natural^p} A c_\alpha \theta^\alpha.
\]
By linear independence of the monomials, the coefficients of each $\theta^\alpha$ must match, which gives
\[
    (A-\lambda^\alpha I) c_\alpha=0 \quad \text{for all } \alpha\in\natural^p.
\]
In addition, by the implicit function theorem, there exists an optimizer $z_*(\theta)$ such that $\nabla_z f(z_*(\theta),\theta)\equiv 0$. Since the optimizer is unique, the output condition requires
\[
    C \sum_\alpha c_\alpha y^\alpha = C \tilde\sigma(y) = z_*(y) = \sum_\alpha d_\alpha y^\alpha,
\]
where $d_\alpha\in\real$ are the coefficients in the series expansion of $z_*(y)$. Again matching coefficients gives $C c_\alpha = d_\alpha$. To summarize, a mapping $\sigma$ with series expansion~\eqref{eq:series} satisfies the internal model conditions if and only if $(A-\lambda^\alpha I)c_\alpha = 0$ and ${C c_\alpha = d_\alpha}$ for each multi-index $\alpha\in\natural^p$. This insight leads to the following result on the algorithm structure.

\begin{lemma}[Algorithm structure]\label{lemma:structure}
    Fix the objective parameters $\mu$ and $L$, exosystem model $S$, and algorithm matrices $(A,B,C)$. Suppose $S$ is diagonalizable with eigenvalues $\lambda=(\lambda_1,\ldots,\lambda_p)\in\complex^p$. Then, the following are equivalent:
    \begin{itemize}
        \item[(i)] There exists an analytic mapping $\sigma : \real^p\to\real^n$ satisfying the internal model kinematic and output equations:
        \begin{equation*}
        \sigma(S\theta) = A \sigma(\theta) \quad\text{and}\quad
            \nabla_z f(C \sigma(\theta), \theta) = 0.
        \end{equation*}
        \item[(ii)] For each harmonic $\omega\in\Omega_\lambda$, $A$ has an eigenvalue at~$\omega$ and the following non-resonance condition holds:
        \begin{equation*}
            \bmat{A-\omega I & B \\ C & 0} \text{ full rank}. \tag*{\QEDB}
        \end{equation*}
    \end{itemize}
\end{lemma}

\begin{proof}
    (ii) $\Rightarrow$ (i). From (ii), there exists an eigenvector $v_\alpha\in\complex^n$ of $A$ with eigenvalue $\lambda^\alpha\in\complex$ for each ${\alpha\in\natural^p}$. Moreover, non-resonance implies that the product
    \[
        \bmat{A-\lambda^\alpha I & B \\ C & 0} \bmat{v_\alpha \\ 0}
    \]
    is nonzero, so $C v_\alpha\neq 0$. From our previous discussion, the mapping $\sigma$ of the form~\eqref{eq:series} with $c_\alpha = d_\alpha (C v_\alpha)^{-1} v_\alpha$ satisfies the internal model conditions, so (i) holds.

    (i) $\Rightarrow$ (ii). Suppose (i) holds, and consider the series expansion of~$\sigma$ in~\eqref{eq:series}. Suppose that, for some $\alpha$, the optimizer contains mode $y^\alpha$, meaning that $d_\alpha\neq 0$. The conditions on $c_\alpha$ imply that it is an eigenvector of $A$ with eigenvalue $\omega = \lambda^\alpha$. Suppose the non-resonance condition fails with this~$\alpha$. Then there exists a nonzero pair $(p,q)$ such that
    \[
        \omega p^\tp = p^\tp A + q^\tp C \qquad\text{and}\qquad
        0 = p^\tp B.
    \]
    Multiplying the kinematic internal model condition on the left by $p^\tp$ and using this first equation yields
    \[
        p^\tp \sigma(S\theta) = \omega p^\tp \sigma(\theta) - q^\tp z^*(\theta),
    \]
    where $z^*(\theta) = C \sigma(\theta)$ is the optimizer. Now consider a trajectory $\theta_k$ of the exosystem and define $y_k \defeq p^\tp \sigma(\theta_k)$. This signal evolves according to the linear difference equation
    \[
        y_{k+1} - \omega y_k = - q^\tp z^*(\theta_k),
    \]
    which has natural frequency $\omega$. We can choose $f$ such that its optimal trajectory is such that the forcing term $-q^\tp z^*(\theta_k)$ is nonzero at this frequency, which then resonates with the system to grow unbounded. But $\sigma$ is continuous from (i) and the exosystem trajectory is bounded from \Cref{assumption:exosystem}, so $y_k$ must be bounded --- a contradiction. Therefore, the non-resonance condition is satisfied, so (ii) holds.\hfill
\end{proof}

Item (ii) of \Cref{lemma:structure} enforces that the algorithm transfer function $G(z)$ must have a pole at each harmonic $\omega\in\Omega_\lambda$. Therefore, any such transfer function must factor as~\eqref{eq:structure} for some proper rational transfer function $K(z)$. With this structure, \Cref{prob:main} becomes that of designing the controller $K(z)$ such that~\eqref{eq:alg} is robustly exponentially stable for all time-invariant objective functions that satisfy \Cref{assumption:objective}, from which the claim in \Cref{thm:main} follows.

{\bibliographystyle{IEEEtran}
\footnotesize
\bibliography{BIB/brevalias,BIB/references,BIB/full_GB,BIB/GB}}

\end{document}